\documentclass[12pt]{amsart}
\usepackage{}
\usepackage{amsmath}
\usepackage{amsfonts}
\usepackage{amssymb}
\usepackage[all]{xy}           
\usepackage{bm}
\usepackage{bbding}
\usepackage{txfonts}
\usepackage{amscd}
\usepackage{xspace}
\usepackage[shortlabels]{enumitem}
\usepackage{ifpdf}

\ifpdf
  \usepackage[colorlinks,final,backref=page,hyperindex]{hyperref}
\else
  \usepackage[colorlinks,final,backref=page,hyperindex,hypertex]{hyperref}
\fi
\usepackage{tikz}
\usepackage[active]{srcltx}

\makeatletter

\newtheorem{thm}{Theorem}[section]
\newtheorem{lem}[thm]{Lemma}
\newtheorem{cor}[thm]{Corollary}
\newtheorem{pro}[thm]{Proposition}
\newtheorem{ex}[thm]{Example}
\theoremstyle{definition}
\newtheorem{rmk}[thm]{Remark}
\newtheorem{defi}[thm]{Definition}

\newcommand{\nc}{\newcommand}
\newcommand{\delete}[1]{}

\nc{\mlabel}[1]{\label{#1}}  
\nc{\mcite}[1]{\cite{#1}}  
\nc{\mref}[1]{\ref{#1}}  
\nc{\mbibitem}[1]{\bibitem{#1}} 

\delete{
\nc{\mlabel}[1]{\label{#1}{\hfill \hspace{1cm}{\bf{{\ }\hfill(#1)}}}}
\nc{\mcite}[1]{\cite{#1}{{\em{{\ }(#1)}}}}  
\nc{\mref}[1]{\ref{#1}{{\em{{\ }(#1)}}}}  
\nc{\mbibitem}[1]{\bibitem[\em #1]{#1}} 
}

\newcommand {\emptycomment}[1]{}

\nc{\oprn}{\theta}
\nc{\Oprn}{\Theta}

\nc{\calo}{\mathcal{O}}
\nc{\oop}{$\mathcal{O}$-operator\xspace}
\nc{\oops}{$\mathcal{O}$-operators\xspace}
\nc{\mrho}{{\bm{\varrho}}}
\nc{\emk}{\mathbf{K}}
\nc{\invlim}{\displaystyle{\lim_{\longleftarrow}}\,}
\nc{\ot}{\otimes}

\newcommand{\lon }{\,\rightarrow\,}
\newcommand{\be }{\begin{equation}}
\newcommand{\ee }{\end{equation}}

\newcommand{\g}{\mathfrak g}
\newcommand{\h}{\mathfrak h}

\newcommand{\Real}{\mathbb R}
\newcommand{\Comp}{\mathbb C}

\newcommand{\huaB}{\mathcal{B}}

\newcommand{\huaL}{\mathcal{L}}
\newcommand{\huaR}{\mathcal{R}}

\newcommand{\huaG}{\mathcal{G}}

\newcommand{\huaW}{\mathcal{W}}
\newcommand{\huaX}{\mathcal{X}}
\newcommand{\huaY}{\mathcal{Y}}

\newcommand{\huaO}{{\mathcal{O}}}

\newcommand{\huaZ}{\mathcal{Z}}

\newcommand{\frkg}{\mathfrak g}

\newcommand{\frkB}{\mathfrak B}

\newcommand{\frkL}{\mathfrak L}

\newcommand{\frkR}{\mathfrak R}

\newcommand{\pair}[1]{\left\langle #1\right\rangle}

\newcommand{\Courant}[1]{\left\llbracket  #1\right\rrbracket }

\newcommand{\Id}{{\rm{Id}}}

\newcommand{\br}[1]{   [ \cdot,    \cdot  ]   }

\newcommand{\Hom}{\mathrm{Hom}}

\newcommand{\LYP}{$\mathsf{LYRep}$~}

\newcommand{\gl}{\mathfrak {gl}}

\newcommand{\ad}{\mathrm{ad}}

\newcommand{\LYA}{Lie-Yamaguti algebra}

\newcommand{\ltp}{\Courant{\cdot,\cdot,\cdot}}
\newcommand{\Rea}{\rm{Re}}

\nc{\CV}{\mathbf{C}}

\begin{document}

\title[Para-K\"ahler and pseudo-K\"ahler structures on Lie-Yamaguti algebras]{Para-K\"ahler and pseudo-K\"ahler structures on Lie-Yamaguti algebras}

\author[Jia Zhao]{$\textsc{Jia Zhao}^1$}
\address{Jia Zhao, School of Sciences, Nantong University, Nantong, 226019, Jiangsu, China}
\email{zhaojia@ntu.edu.cn}

\author[Yuqin Feng]{$\textsc{Yuqin Feng}^{2}$}
\address{Yuqin Feng, School of Mathematics and Statistics, Shaanxi Normal University, Xi'an 710119, Shaanxi, China}
\email{feng@snnu.edu.cn}

\author[Yu Qiao]{$\textsc{Yu Qiao}^{2,*}$}
\address{Yu Qiao (corresponding author), School of Mathematics and Statistics, Shaanxi Normal University, Xi'an 710119, Shaanxi, China}
\email{yqiao@snnu.edu.cn}


\begin{abstract}
For a pre-\LYA~$A$, by using its sub-adjacent \LYA~$A^c$, we are able to construct a semidirect product \LYA ~via a representation of $A^c$. The investigation of such semidirect \LYA s leads us to the notions of para-K\"ahler structures and pseudo-K\"ahler structures on \LYA s, and also gives the definition of complex product structures on \LYA s which was first introduced in \cite{Sheng Zhao}.
Furthermore, a Levi-Civita product with respect to a pseudo-Riemannian \LYA~ is introduced and we explore its relation with pre-\LYA s.

\end{abstract}

\thanks{{\em Mathematics Subject Classification} (2020): 17A30, 17A60, 17B99}
\thanks{ {\em keywords}: Lie-Yamaguti algebra, complex product structure, para-K\"ahler structure, pseudo-K\"ahler structure}
\thanks{1: \emph{School of Sciences, Nantong University, Nantong 226019, Jiangsu, China}}
\thanks{2: {\em School of Mathematics and Statistics, Shaanxi Normal University, Xi'an 710119, Shaanxi, China}}
\thanks{*: {Corresponding author}}
\thanks{\emph{Emails}: zhaojia@ntu.edu.cn,~feng@snnu.edu.cn, ~yqiao@snnu.edu.cn}

\maketitle

\vspace{-1.1cm}

\tableofcontents

\allowdisplaybreaks

 \section{Introduction}
An almost product structure on a Lie algebra $\g$ is a linear map $E:\g\longrightarrow\g$ such that $E^2=\Id$. If in addition, $E$ also satisfies the following integrability condition:
$$[Ex,Ey]=E\Big([Ex,y]+[x,Ey]-E[x,y]\Big),\quad\forall x,y \in \g,$$
then $E$ is a product structure on $\g$. Note that there exists a product structure on a Lie algebra $\g$ if and only if $\g$ can be decomposed as a direct sum of two subalgebras: $\g=\g_1\oplus\g_2$. Parallel to product structures, an almost complex structure is a linear map $J:\g\longrightarrow\g$ such that $J^2=-\Id$. A complex structure on $\g$ is also an almost complex structure such that the integrability condition is satisfied. A complex product structure is a pair $(J,E)$, where $J$ is a complex structure and $E$ is a product structure such that a certain condition is satisfied. One can read \cite{A A2,A S,T.S} for more details about complex product structures on Lie or $3$-Lie algebras. A symplectic structure on $\g$ is a nondegenerate $2$-cocycle $\omega\in \wedge^2\g^*$ \cite{BBM}. A para-K\"ahler structure is a pair $(\omega,E)$, where $\omega$ is a symplectic structure and  $E$ is a paracomplex structure (a product structure such that the decomposed two subalgebras have the same dimension) such that a certain condition is satisfied. Moreover, abelian para-K\"{a}hler Lie algebras have been investigated in literature \cite{B. B1,B. B3,G.C2}. Parallel to the para-K\"ahler structure, a pseudo-K\"ahler structure is a pair $(\omega,J)$, where $\omega$ is a symplectic structure and $J$ is a complex structure such that a certain condition is satisfied. See \cite{A A1,BFLM,T.S} for more details about pseudo-K\"ahler structures on Lie algebras or on $3$-Lie algebras. Moreover, pseudo-metric Riemannian Lie algebras were investigated in \cite{Chen Hou Bai}. These structures have many applications in mathematics, geometry, and mathematical physics. As we all know, complex product structures, para-K\"ahler structures, and pseudo-K\"ahler structures are closely related with pre-Lie algebras, which are the underlying algebra structures of relative Rota-Baxter operators (also called $\huaO$-operators or Kupershmidt operators). Bai and his collaborators explored several properties about pre-Lie algebras (also called left-symmetric algebras) and studied their relation with the classical Yang-Baxter equation. See \cite{BCM3,BCM1,BCM2,B.G.S} for more details about pre-Lie algebras or $3$-pre-Lie algebras and see \cite{Gub,Kupershmidt} for relative Rota-Baxter operators and Rota-Baxter algebras.

A \LYA ~is a generalization of Lie algebras and Lie triple systems, and can be dated back to Nomizu's work on invariant affine connections on homogeneous spaces in 1950's (\cite{Nomizu}) and Yamaguti's work on general Lie triple systems and Lie triple algebras (\cite{Yamaguti1}). Its representation and cohomology theory were constructed in \cite{Yamaguti2,Yamaguti3} during 1950's to 1960's. Later until 21st century, Kinyon and Weinstein named this object as a \LYA ~in their study of Courant algebroids in \cite{Weinstein}. \LYA s have attracted much attention in recent years. For instance, Benito, Draper, and Elduque investigated \LYA s related to simple Lie algebras of type$G_2$  \cite{B.D.E}. Afterwards, Benito, Elduque, and Mart$\acute{i}$n-Herce explored irreducible \LYA s in \cite{B.E.M1,B.E.M2}. Furthermore, Benito, Bremmer, and Madariaga examined orthogonal \LYA s in \cite{B.B.M}. Recently, we studied cohomology and deformations of relative Rota-Baxter operators on \LYA s \cite{ZQ1}, relative Rota-Baxter-Nijenhuis structures on a \LYA ~with a representation \cite{ZQ2}, and bialgebra theory of \LYA s \cite{ZQ3}.

Moreover, Sheng and the first author explored product structures and complex structures on \LYA s in \cite{Sheng Zhao} and deeply examined relative Rota-Baxter operators, symplectic structures, and pre-\LYA s in \cite{SZ1}. This motivates us to consider compatibility conditions between a product structure and a symplectic structure, and between a complex structure and a symplectic structure, which in turn lead us to introduce the notions of para-K\"ahler structure and pseudo-K\"ahler structure on a \LYA.
Thus this paper can be regarded as a sequel to \cite{SZ1} and \cite{Sheng Zhao}.

Parallel to the context of Lie algebras, the notion of para-K\"ahler structures of a \LYA ~is obtained via a paracomplex structure and a symplectic structure. An equivalent description of a para-K\"ahler structure is given by the decomposition of the original \LYA. With respect to a para-K\"ahler structure, there exists a pseudo-Riemannian structure. We define Levi-Civita product with respect to a pseudo-Riemannian \LYA, and give its precise formula. Finally, we add a compatibility condition between a complex structure and a symplectic structure to introduce the notion of pseudo-K\"ahler structures on \LYA s. The relation between a para-K\"ahler structure and a pseudo-K\"ahler structure is investigated. Furthermore, for a pre-\LYA ~$A$, by using its sub-adjacent \LYA ~$A^c$, we obtain the semidirect product \LYA ~$A^c\ltimes_{L^*,-\huaR^*\tau}A^*$ via the representation $(A^*;L^*,-\huaR^*\tau)$ of $A^c$. Following \cite{Sheng Zhao}, we construct a perfect complex product structure on the semidirect product \LYA s~$A^c\ltimes_{L^*,-\huaR^*\tau}A^*$ and $A^c\ltimes_{L,\huaR}A$ respectively, and build further a para-K\"ahler structure and a pseudo-K\"ahler structure on $A^c\ltimes_{L^*,-\huaR^*\tau}A^*$ respectively.

This paper is structured as follows. In Section 2, we recall some basic definitions. In Section 3, we construct a complex product structure on a larger \LYA ~from a pre-\LYA. In Section 4, we introduce the notion of para-K\"ahler structures and give its equivalent descriptions. Furthermore, the notion of Levi-Civita products with respect to a pseudo-Riemannian \LYA ~is introduced and we show that Levi-Civita product coincides with the pre-\LYA ~structure. In Section 5, we introduce the notion of pseudo-K\"ahler structures and study its relation with para-K\"ahler structures.

In this paper, all the vector spaces are over $\mathbb{K}$, a field of characteristic $0$.

\smallskip

{\bf Acknowledgements:} We would like to thank Professor Yunhe Sheng of Jilin University for useful discussions. Qiao was partially supported by NSFC grant 11971282.

\smallskip
\section{Preliminaries}
In this section, we first recall some basic notions such as Lie-Yamaguti algebras and representations, which are main objects throughout this paper. 

\begin{defi}\cite{Yamaguti1}\label{LY}
A {\bf Lie-Yamaguti algebra} is a vector space $\g$, together with a bilinear bracket $[\cdot,\cdot]:\wedge^2  \mathfrak{g} \to \mathfrak{g} $ and a trilinear bracket $\Courant{\cdot,\cdot,\cdot}:\wedge^2\g \otimes  \mathfrak{g} \to \mathfrak{g} $ such that the following equations are satisfied for all $x,y,z,w,t \in \g$,
\begin{eqnarray}
~ &&\label{LY1}[[x,y],z]+[[y,z],x]+[[z,x],y]+\Courant{x,y,z}+\Courant{y,z,x}+\Courant{z,x,y}=0,\\
~ &&\Courant{[x,y],z,w}+\Courant{[y,z],x,w}+\Courant{[z,x],y,w}=0,\\
~ &&\Courant{x,y,[z,w]}=[\Courant{x,y,z},w]+[z,\Courant{x,y,w}],\\
~ &&\Courant{x,y,\Courant{z,w,t}}=\Courant{\Courant{x,y,z},w,t}+\Courant{z,\Courant{x,y,w},t}+\Courant{z,w,\Courant{x,y,t}}.\label{fundamental}
\end{eqnarray}
\end{defi}

\begin{rmk}
If the binary bracket $[\cdot,\cdot]=0$, then a Lie-Yamaguti algebra reduces to a Lie triple system;
If the ternary bracket $\Courant{\cdot,\cdot,\cdot}=0$, then a Lie-Yamaguti algebra reduces to a Lie algebra.
\end{rmk}

\begin{ex}
Let $(\frkg,[\cdot,\cdot])$ be a Lie algebra. We define a trilinear bracket $\Courant{\cdot,\cdot,\cdot
 }:\wedge^2\g\otimes \g\lon \g$ by  $$\Courant{x,y,z}:=[[x,y],z],\quad \forall x,y, z \in \mathfrak{g}.$$  Then $(\g,[\cdot,\cdot],\Courant{\cdot,\cdot,\cdot})$ becomes a Lie-Yamaguti algebra naturally.
\end{ex}

More examples can be found in \cite{B.D.E}.
\emptycomment{
\begin{defi}\cite{Sheng Zhao,Takahashi}
Let $(\g,[\cdot,\cdot]_{\g},\Courant{\cdot,\cdot,\cdot}_{\g})$ and $(\h,[\cdot,\cdot]_{\h},\Courant{\cdot,\cdot,\cdot}_{\h})$ be two Lie-Yamaguti algebras. A {\bf homomorphism} from $(\g,[\cdot,\cdot]_{\g},\Courant{\cdot,\cdot,\cdot}_{\g})$ to $(\h,[\cdot,\cdot]_{\h},\Courant{\cdot,\cdot,\cdot}_{\h})$ is a linear map $\phi:\g \to \h$ such that for all $x,y,z \in \g$
\begin{eqnarray*}
\phi([x,y]_{\g})&=&[\phi(x),\phi(y)]_{\h},\\
~ \phi(\Courant{x,y,z}_{\g})&=&\Courant{\phi(x),\phi(y),\phi(z)}_{\h}.
\end{eqnarray*}
\end{defi}}

\begin{defi}\cite{Yamaguti2}
Let $(\g,[\cdot,\cdot],\Courant{\cdot,\cdot,\cdot})$ be a Lie-Yamaguti algebra and $V$ a vector space. A {\bf representation} of $\g$ on $V$ consists of a linear map $\rho:\g \to \gl(V)$ and a bilinear map $\mu:\otimes^2 \g \to \gl(V)$ such that for all $x,y,z,w \in \g$
\begin{eqnarray*}
~&&\label{RLYb}\mu([x,y],z)-\mu(x,z)\rho(y)+\mu(y,z)\rho(x)=0,\\
~&&\label{RLYd}\mu(x,[y,z])-\rho(y)\mu(x,z)+\rho(z)\mu(x,y)=0,\\
~&&\label{RLYe}\rho(\Courant{x,y,z})=[D_{\rho,\mu}(x,y),\rho(z)],\\
~&&\label{RYT4}\mu(z,w)\mu(x,y)-\mu(y,w)\mu(x,z)-\mu(x,\Courant{y,z,w})+D_{\rho,\mu}(y,z)\mu(x,w)=0,\\
~&&\label{RLY5}\mu(\Courant{x,y,z},w)+\mu(z,\Courant{x,y,w})=[D_{\rho,\mu}(x,y),\mu(z,w)],
\end{eqnarray*}
where $D_{\rho,\mu}:\otimes^2A \to \gl(V)$ is defined to be
 \begin{eqnarray}
 D_{\rho,\mu}(x,y):=\mu(y,x)-\mu(x,y)+[\rho(x),\rho(y)]-\rho([x,y]), \quad \forall x,y, \in A.\label{rep}
 \end{eqnarray}
It is obvious that $D_{\rho,\mu}$ is skew-symmetric.
We denote a representation of $\g$ on $V$ by $(V;\rho,\mu)$.
\end{defi}

\begin{rmk}
If a Lie-Yamaguti algebra reduces to a Lie tripe system (then $[\cdot,\cdot]=0$), we get the notion of representation of the Lie triple system $(\g,\Courant{\cdot,\cdot,\cdot})$ on $V$ (then $\rho=0$).
\end{rmk}

\begin{pro}
Let $(V;\rho,\mu)$ be a representation of a \LYA ~$(\g,[\cdot,\cdot],\Courant{\cdot,\cdot,\cdot})$, then the following equalities hold for all $x,y,z,w\in \g$
\begin{eqnarray*}
\label{RLYc}D_{\rho,\mu}([x,y],z)+D_{\rho,\mu}([y,z],x)+D_{\rho,\mu}([z,x],y)=0,\\
\label{RLY5a}D_{\rho,\mu}(\Courant{x,y,z},w)+D_{\rho,\mu}(z,\Courant{x,y,w})=[D_{\rho,\mu}(x,y),D_{\rho,\mu}(z,w)],\\
\mu(\Courant{x,y,z},w)=\mu(x,w)\mu(z,y)-\mu(y,w)\mu(z,x)-\mu(z,w)D_{\rho,\mu}(x,y).\label{RLY6}
\end{eqnarray*}
\end{pro}

\begin{ex}\label{ad}
Let $(\g,[\cdot,\cdot],\Courant{\cdot,\cdot,\cdot})$ be a Lie-Yamaguti algebra. We define $\ad:\g \to \gl(V)$ and $\frkR :\otimes^2\g \to \gl(V)$ by $x \mapsto \ad_x$ and $(x,y) \mapsto \mathfrak{R}_{x,y}$ respectively, where $\ad_xz=[x,z]$ and $\mathfrak{R}_{x,y}z=\Courant{z,x,y}$ for all $z \in \g$. Then $(\ad,\mathfrak{R})$ forms a representation of $\g$ on itself, called the {\bf adjoint representation}. In this case, by \eqref{rep}, $\frkL\triangleq D_{\ad,\frkR}$ is given by
\begin{eqnarray*}
\frkL_{x,y}=\mathfrak{R}_{y,x}-\mathfrak{R}_{x,y}-\ad_{[x,y]}+[\ad_x,\ad_y].
\end{eqnarray*}
Moreover, by \eqref{LY1} and \eqref{rep}, we have
\begin{eqnarray}
\frkL_{x,y}z=\Courant{x,y,z},\quad\forall z\in \g.\label{lef}
\end{eqnarray}
\end{ex}

Representations of a Lie-Yamaguti algebra can be characterized by the semidirect product Lie-Yamaguti algebras.
\begin{pro}\cite{Zhang1,Zhang2}
Let $(\g,[\cdot,\cdot],\Courant{\cdot,\cdot,\cdot})$ be a Lie-Yamaguti algebra and $V$ a vector space. If $\rho:\g \to \gl(V)$ and $\mu:\otimes^2 \g \to \gl(V)$ are linear and bilinear map respectively, then $(V;\rho,\mu)$ is a representation of $(\g,[\cdot,\cdot],\Courant{\cdot,\cdot,\cdot})$ if and only if $\g \oplus V$ endows with a Lie-Yamaguti algebra structure under for all $x,y,z \in \g, u,v,w \in V$
\begin{eqnarray*}
[x+u,y+v]_{\rho,\mu}&=&[x,y]+\rho(x)v-\rho(y)u,\\
~\Courant{x+u,y+v,z+w}_{\rho,\mu}&=&\Courant{x,y,z}+D_{\rho,\mu}(x,y)w+\mu(y,z)u-\mu(x,z)v,
\end{eqnarray*}
where $D_{\rho,\mu}$ is given by \eqref{rep}.
The Lie-Yamaguti algebra $(\g \oplus V,[\cdot,\cdot]_{\rho,\mu},\Courant{\cdot,\cdot,\cdot}_{\rho,\mu})$ is called the {\bf semidirect product Lie-Yamaguti algebra}, denoted by $\g \ltimes_{\rho,\mu} V$.
\end{pro}
\emptycomment{
\begin{defi}
Let $(\g,[\cdot,\cdot]_{\g},\Courant{\cdot,\cdot,\cdot}_{\g})$ and $(\h,[\cdot,\cdot]_{\h},\Courant{\cdot,\cdot,\cdot}_{\h})$ be two Lie-Yamaguti algebras, and $(V;\rho,\mu)$ and $(W;\varrho,\nu)$ be representations of $\g$ and $\h$ respectively. A {\bf homomorphism} from $(V;\rho,\mu)$ to $(W;\varrho,\nu)$ is a pair $(\phi,\psi)$ that consists of a Lie-Yamaguti algebra homomorphism $\phi$ from $\g$ to $\h$ and a linear map $\psi:V \to W$ such that for all $x,y \in \g, v \in V$
\begin{eqnarray*}
~\psi(\rho(x)v)&=&\varrho(\phi(x))\psi(v),\\
~\psi(\mu(x,y)v)&=&\nu(\phi(x),\phi(y))\psi(v).
\end{eqnarray*}
\end{defi}
\begin{rmk}
It is easy to see that
\begin{eqnarray*}
\psi(D_{\rho,\mu}(x,y)v)=D_{\varrho,\nu}(\phi(x),\phi(y))\psi(v),
\end{eqnarray*}
where $D_{\varrho,\nu}$ is defined by \eqref{rep}. More precisely, for any $u,v \in \h$
$$D_{\varrho,\nu}(u,v):=\nu(v,u)-\nu(u,v)+[\varrho(u),\varrho(v)]-\varrho([u,v]_{\h}).$$
\end{rmk}}

Let $(V;\rho,\mu)$ be a representation of a \LYA ~$(\g,[\cdot,\cdot],\Courant{\cdot,\cdot,\cdot})$. We define linear maps $\rho^*:\g \to \gl(V^*)$ and $\mu^*:\otimes^2 \g \to \gl(V^*)$ to be
\begin{eqnarray*}
\langle\rho^*(x)\alpha, v \rangle&=&-\langle\alpha,\rho(x)v\rangle,\\
\langle\mu^*(x,y)\alpha,v\rangle&=&-\langle\alpha,\mu(x,y)v\rangle,
\end{eqnarray*}
for all $x,y \in \g, ~\alpha \in V^*, v \in V$.

Let $V$ be a vector space. Define the switching operator $\tau: \otimes^2V \to \otimes^2V$ by
\begin{eqnarray*}
\tau(x\otimes y)=y\otimes x, \quad \forall x\otimes y \in \otimes^2V.
\end{eqnarray*}

In \cite{SZ1}, we have constructed the dual representation of a Lie-Yamaguti algebra.

\begin{pro}\label{dual}
Let $(\g,[\cdot,\cdot],\Courant{\cdot,\cdot,\cdot})$ be a Lie-Yamaguti algebra with a representation $(V;\rho,\mu)$. Then
\begin{eqnarray*}
\big(V^*;\rho^*,-\mu^*\tau\big)
\end{eqnarray*}
is a representation of Lie-Yamaguti algebra $(\g,[\cdot,\cdot],\Courant{\cdot,\cdot,\cdot})$ on $V^*$, which is called the {\bf dual representation} of $\g$. Here $D_{\rho,\mu}^*=D_{\rho^*,-\mu^*\tau}$.
\end{pro}

\begin{ex}
Let $(\g;\ad,\mathfrak{R})$ be the adjoint representation of a Lie-Yamaguti algebra $(\g,[\cdot,\cdot],\Courant{\cdot,\cdot,\cdot})$, where $\ad$ and $\mathfrak{R}$ are given in Example \ref{ad}. Then $(\g^*;\ad^*,-\mathfrak{R}^*{\tau})$ is the dual representation of the adjoint representation, called the {\bf coadjoint representation}. Note that $\mathfrak{L}^*$ is the dual of $\mathfrak{L}$.
\end{ex}

Next, let us recall the notion of pre-\LYA s introduced in \cite{SZ1}.
\begin{defi}{\rm (\cite{SZ1})}
A {\bf pre-Lie-Yamaguti algebra} is a vector space $A$ with a bilinear operation $*:\otimes^2A \to A$ and a trilinear operation $\{\cdot,\cdot,\cdot\} :\otimes^3A \to A$ such that for all $x,y,z,w,t \in A$
\begin{eqnarray*}
~ &&\label{pre2}\{z,[x,y]_C,w\}-\{y*z,x,w\}+\{x*z,y,w\}=0,\\
~ &&\label{pre4}\{x,y,[z,w]_C\}=z*\{x,y,w\}-w*\{x,y,z\},\\
~ &&\label{pre5}\{\{x,y,z\},w,t\}-\{\{x,y,w\},z,t\}-\{x,y,\{z,w,t\}_D\}-\{x,y,\{z,w,t\}\}\\
~ &&\nonumber~~~+\{x,y,\{w,z,t\}\}+\{z,w,\{x,y,t\}\}_D=0,\\
~ &&\label{pre6}\{z,\{x,y,w\}_D,t\}+\{z,\{x,y,w\},t\}-\{z,\{y,x,w\},t\}+\{z,w,\{x,y,t\}_D\}\\
~ &&\nonumber~~~+\{z,w,\{x,y,t\}\}-\{z,w,\{y,x,t\}\}=\{x,y,\{z,w,t\}\}_D-\{\{x,y,z\}_D,w,t\},\\
~&&\label{pre7}\{x,y,z\}_D*w+\{x,y,z\}*w-\{y,x,z\}*w=\{x,y,z*w\}_D-z*\{x,y,w\}_D,
\end{eqnarray*}
where the commutator
$[\cdot,\cdot]_C:\wedge^2\g \to \g$ and $\{\cdot,\cdot,\cdot\}_D: \otimes^3 A \to A$ are defined by for all $x,y,z \in A$,
\begin{eqnarray}
~[x,y]_C:=x*y-y*x, \quad \forall x,y \in A,\label{pre10}
\end{eqnarray}
and
\begin{eqnarray}
\{x,y,z\}_D:=\{z,y,x\}-\{z,x,y\}+(y,x,z)-(x,y,z), \label{pre3}
\end{eqnarray}
respectively. Here $(\cdot,\cdot,\cdot)$ denotes the associator: $(x,y,z):=(x*y)*z-x*(y*z)$. It is obvious that $\{\cdot,\cdot,\cdot\}_D$ is skew-symmetric with respect to the first two variables. We denote a pre-Lie-Yamaguti algebra by $(A,*,\{\cdot,\cdot,\cdot\})$.
\end{defi}

Let $(A,*,\{\cdot,\cdot,\cdot\})$ be a pre-Lie-Yamaguti algebra. Define
\begin{itemize}
\item a ternary operation $\ltp_C$ to be
\begin{eqnarray}
~\label{subLY2}\Courant{x,y,z}_C&=&\{x,y,z\}_D+\{x,y,z\} -\{y,x,z\}, \quad \forall x,y,z \in \g,
\end{eqnarray}
where $\{\cdot,\cdot,\cdot\}_D$ is given by \eqref{pre3}.
\item linear maps
$$L: A \to \gl(A), \quad \huaR: \otimes^2A \to \gl(A)$$
 to be
 $$x \mapsto L_x, \quad (x,y)\mapsto \huaR(x,y)$$
 respectively, where $L_xz=x*z$ and $R(x,y)z=\{z,x,y\}$ for all $z \in A$.
\end{itemize}

\begin{pro}\label{subad}{\rm (\cite{SZ1})}
With the above notations, then  we have
\begin{itemize}
\item [\rm (i)] the operation $(\br_{_C},\ltp_C)$ defines a Lie-Yamaguti algebra structure on $A$, where $\br__C$ and $\ltp_C$ are give by Eqs. \eqref{pre10} and \eqref{subLY2} respectively. This Lie-Yamaguti algebra $(A,\br_{_C},\\
    \ltp_C)$ is called the {\bf sub-adjacent \LYA} and is denoted by $A^c$;
\item [\rm (ii)] the triple $(A;L,\huaR)$ is a representation of the sub-adjacent Lie-Yamaguti algebra $A^c$ on $A$. Furthermore, the identity map $\Id:A\longrightarrow A$ is a relative Rota-Baxter operator on $A^c$ with respect to the representation $(A;L,\huaR)$, where
    $$\huaL:=D_{L,\huaR}:\wedge^2A\longrightarrow \gl(A),\quad (x,y)\mapsto L(x,y)$$
    is given by
    $$\huaL(x,y)z=\{x,y,z\}_D,\quad\forall z\in A.$$
\end{itemize}
\end{pro}

A {\bf quadratic Lie-Yamaguti algebra} \cite{Kikkawa} is a Lie-Yamaguti algebra $(\g,[\cdot,\cdot],\Courant{\cdot,\cdot,\cdot})$ equipped with a nondegenerate symmetric bilinear form $\huaB \in \otimes^2\g^*$ satisfying the following invariant conditions for all $x,y,z,w\in \g$
\begin{eqnarray*}
\label{invr1}\huaB([x,y],z)&=&-\huaB(y,[x,z]),\\
\label{invr2}\huaB(\Courant{x,y,z},w)&=&\huaB(x,\Courant{w,z,y}).
\end{eqnarray*}
Then $\huaB$ induces an isomorphism
$$\huaB^\sharp:\g \to \g^*$$
defined by
\begin{eqnarray}
\langle\huaB^\sharp(x),y\rangle=\huaB(x,y), \quad \forall x,y \in \g.\label{invariant}
\end{eqnarray}
\begin{rmk}
Note that
\begin{eqnarray*}
\huaB(x,\Courant{w,z,y})=-\huaB(x,\Courant{z,w,y})=-\huaB(\Courant{z,w,y},x)=-\huaB(z,\Courant{x,y,w}).
\end{eqnarray*}
Thus we obtain
\begin{eqnarray}
\label{invr3}\huaB(\Courant{x,y,z},w)=-\huaB(z,\Courant{x,y,w}).
\end{eqnarray}
\end{rmk}

\begin{pro}{\rm(\cite{SZ1})}
With the above notations, $(\Id,\huaB^\sharp)$ forms an isomorphism from the adjoint representation $(\g;\ad,\mathfrak{R})$ to the coadjoint representation $(\g^*;\ad^*,-\mathfrak{R}^*{\tau})$.
\end{pro}
\emptycomment{
\begin{proof}
For all $x,y \in \g$, by \eqref{invr1} and \eqref{invariant}, we have
\begin{eqnarray*}
\langle\huaB^\sharp(\ad_xy),z\rangle=\langle\huaB^\sharp([x,y]),z\rangle=\huaB([x,y],z)=-\huaB(y,[x,z])=-\langle\huaB^\sharp(y),[x,z]\rangle
=\langle\ad_x^*\huaB^\sharp(y),z\rangle.
\end{eqnarray*}
By the arbitrary of $z$, we have
\begin{eqnarray}
\huaB^\sharp(\ad_xy)=\ad_x^*\huaB^\sharp(y).\label{iso1}
\end{eqnarray}
Furthermore, for all $x,y,z,w \in \g$, by \eqref{invariant} and \eqref{invr2}, we have
\begin{eqnarray*}
\langle\huaB^\sharp(\mathfrak{R}_{x,y}z),w\rangle=\huaB(\Courant{z,x,y},w)=\huaB(z,\Courant{w,y,x})=\langle\huaB^\sharp(z),\Courant{w,y,x}\rangle
=-\langle(\mathfrak{R}_{y,x})^*\huaB^\sharp(z),w\rangle.
\end{eqnarray*}
By the arbitrary of $w$, we have
\begin{eqnarray}
\huaB^\sharp(\mathfrak{R}_{x,y}z)&=&-(\mathfrak{R}_{y,x})^*\huaB^\sharp(z).\label{iso3}
\end{eqnarray}
From \eqref{iso1}-\eqref{iso3}, we deduce that $\huaB^\sharp$ is an isomorphism between adjoint representation and coadjoint representation.
This finishes the proof.
\end{proof}}

\begin{rmk}
Similarly, by \eqref{lef}, \eqref{invariant} and \eqref{invr3}, we have
\begin{eqnarray*}
\langle\huaB^\sharp(\mathfrak{L}_{x,y}z),w\rangle=\huaB(\Courant{x,y,z},w)=-\huaB(z,\Courant{x,y,w})=-\langle\huaB^\sharp(z),\Courant{x,y,w}\rangle
=\langle(\mathfrak{L}_{x,y})\huaB^\sharp(z),w\rangle.
\end{eqnarray*}
Thus we obtain that
\begin{eqnarray*}
\huaB^\sharp(\mathfrak{L}_{x,y}z)=(\mathfrak{L}_{x,y})^*\huaB^\sharp(z).
\end{eqnarray*}
\end{rmk}

In the sequel, we recall the notions of symplectic structures, product structures, and complex structures on \LYA s.

\begin{defi}{\rm(\cite{SZ1})}
 Let $(\g,[\cdot,\cdot],\Courant{\cdot,\cdot,\cdot})$ be a Lie-Yamaguti algebra. A {\bf symplectic structure} on $\g$ is a nondegenerate, skew-symmetric bilinear form $\omega \in \wedge^2\g^*$ such that for all $x,y,z,w \in \g$,
 \begin{eqnarray*}
 ~ &&\label{syml}\omega(x,[y,z])+\omega(y,[z,x])+\omega(z,[x,y])=0,\\
 ~ &&\label{sym2}\omega(z,\Courant{x,y,w})-\omega(x,\Courant{w,z,y})+\omega(y,\Courant{w,z,x})-\omega(w,\Courant{x,y,z})=0.
 \end{eqnarray*}
A Lie-Yamaguti algebra $\g$ with a symplectic structure $\omega$ is called a {\bf symplectic Lie-Yamaguti algebra}, denoted by $\Big((\g,[\cdot,\cdot],\Courant{\cdot,\cdot,\cdot}),\omega\Big)$ or $(\g,\omega)$ for short.
\end{defi}

\begin{defi}{\rm{(\cite{Sheng Zhao})}}
Let $(\g,[\cdot,\cdot],\Courant{\cdot,\cdot,\cdot})$ be a Lie-Yamaguti algebra. An {\bf almost product structure} on $\g$ is a linear map $E:\g \longrightarrow\g$ satisfying $E^2 =\Id$ $(E \ne  \pm \Id).$
An almost product structure is called a {\bf product structure} if the following integrability conditions are satisfied:
\begin{eqnarray*}
\label{NI1}[Ex,Ey]&=&E[Ex,y]+E[x,Ey]-[x,y],\\
\nonumber \Courant{Ex,Ey,Ez}&=&E\Courant{Ex,Ey,z}+E\Courant{x,Ey,Ez}+E\Courant{Ex,y,Ez}\\
~ \label{NI2}&&-\Courant{Ex,y,z}-\Courant{x,Ey,z}-\Courant{x,y,Ez}+E\Courant{x,y,z}, \quad \forall x,y,z \in \g.
\end{eqnarray*}
\end{defi}

A product structure gives rise to a decomposition of a Lie-Yamaguti algebra.
\begin{pro}{\rm{(\cite{Sheng Zhao})}}\label{decomp}
Let $(\g,[\cdot,\cdot],\Courant{\cdot,\cdot,\cdot})$ be a Lie-Yamaguti algebra. Then there is a product structure on $\g$ if and only if $\g$ admits a decomposition:
$$\g=\g_+ \oplus \g_-,$$
where $\g_+$ and $\g_-$ are subalgebras of $\g$.
\end{pro}

\begin{defi}
If two decomposed subalgebras $\g_+$ and $\g_-$ have the same dimension, we call the product structure $E$ a {\bf paracomplex structure}. If moreover, the paracomplex structure $E$ is perfect \footnote{For the notion of perfect product structures, one can see Definition 4.12 in \cite{Sheng Zhao}}, $E$ is called a {\bf perfect paracomplex structure}.
\end{defi}

We are able to construct a perfect paracomplex structure on a semidirect product \LYA ~from a pre-\LYA.

\begin{pro}\label{product}
Let $(A,*,\{\cdot,\cdot,\cdot\})$ be a pre-Lie-Yamaguti algebra. Then on the semidirect product Lie-Yamaguti algebra $A^c\ltimes_{L^*,-\huaR^*\tau} A^*$, there is a perfect paracomplex structure $E:A^c\ltimes_{L^*,-\huaR^*\tau} A^* \to A^c\ltimes_{L^*,-\huaR^*\tau} A^*$ given by
\begin{eqnarray}
E(x+\alpha)=x-\alpha, \quad \forall x \in A^c, \alpha \in A^*.\label{prod}
\end{eqnarray}
\end{pro}
\begin{proof}
It is obvious that $E^2=\Id$. Moreover, we have that $(A^c\ltimes_{L^*,-\huaR^*\tau} A^*)_+=A$ and $(A^c\ltimes_{L^*,-\huaR^*\tau} A^*)_-=A^*$ and that they are two subalgebras of the semidirect product Lie-Yamaguti algebra $A^c\ltimes_{L^*,-\huaR^*\tau} A^*$. Thus $E$ is a product structure on $A^c\ltimes_{L^*,-\huaR^*\tau} A^*$. Since $A$ and $A^*$ have the same dimension, $E$ is a paracomplex structure on $A^c\ltimes_{L^*,-\huaR^*\tau} A^*$. It is not hard to deduce that $E$ is perfect.
\end{proof}

\begin{defi}{\rm{(\cite{Sheng Zhao})}}\label{complex}
Let $(\g,[\cdot,\cdot],\Courant{\cdot,\cdot,\cdot})$ be a real Lie-Yamaguti algebra. A linear map $J: \g\longrightarrow\g$ is called {\bf an almost complex structure} if $J^2=-\Id$.
An almost complex structure is called a {\bf complex structure} if the following integrability conditions hold:
\begin{eqnarray}
[Jx,Jy]&=&J[Jx,y]+J[x,Jy]+[x,y],\label{ccom1}\\
~\label{ccom2}J\Courant{x,y,z}&=&-\Courant{Jx,Jy,Jz}+\Courant{Jx,y,z}+\Courant{x,Jy,z}+\Courant{x,y,Jz}\\
~&&\nonumber+J\Courant{Jx,Jy,z}+J\Courant{Jx,y,Jz}+J\Courant{x,Jy,Jz}, \quad \forall x,y,z \in \g.
\end{eqnarray}
\end{defi}

\begin{pro}{\rm{(\cite{Sheng Zhao})}}\label{cdecomp}
Let $(\g,[\cdot,\cdot],\Courant{\cdot,\cdot,\cdot})$ be a real Lie-Yamaguti algebra. Then there is a complex structure on $\g$ if and only if there is a decomposition of $\g_{\mathbb{C}}$:
$$\g_{\mathbb{C}}=\g_i \oplus \g_{-i},$$
where $\g_i=\{x-iJx:x\in \g\}$ and $\g_{-i}=\{x+iJx:x\in \g\}$ are subalgebras of $\g_{\mathbb{C}}$. Observe that $\g_{-i}=\sigma(\g_i)$, where $\g_{\mathbb{C}}$ means the complexification of $\g$, i.e.,
$$\g_{\mathbb{C}}=\g\ot_\Real \Comp\cong\{x+iy:x,y\in \g\},$$
and $\sigma$ means the conjugation map in $\g_{\mathbb{C}}$, i.e.,
$$\sigma (x + iy) = x - iy,\quad\forall x,y \in \g. $$
\end{pro}

\section{Complex product structures on Lie-Yamaguti algebras}
The notion of complex product structures of \LYA s was introduced in \cite{Sheng Zhao}, where we gave an equivalent description of it (see Proposition \ref{complexproduct}). In this section, we construct a perfect complex product structure on a larger \LYA ~from a pre-\LYA.

\begin{defi}
\begin{itemize}
\item [\rm (i)] Let $(\g,[\cdot,\cdot],\Courant{\cdot,\cdot,\cdot})$ be a real Lie-Yamaguti algebra. A {\bf complex product structure} on $\g$ is a pair $(J,E)$ consisting of a product structure $E$ and a complex structure $J$ such that
\begin{eqnarray*}
E\circ J=-J\circ E.
\end{eqnarray*}
\item[\rm (ii)] If $E$ is perfect, we call $(J,E)$ a {\bf perfect complex product structure}.
\end{itemize}
\end{defi}

\begin{pro}\cite{Sheng Zhao}\label{complex1}
Let $(\g,[\cdot,\cdot],\Courant{\cdot,\cdot,\cdot})$ be a complex Lie-Yamaguti algebra. Then $E$ is a product structure on $\g$ if and only if $J=-iE$ is a complex structure on $\g$.
\end{pro}
The following corollary is direct.
\begin{cor}\label{complex2}
Let $J$ be a complex structure on a real Lie-Yamaguti algebra $(\g,[\cdot,\cdot],\Courant{\cdot,\cdot,\cdot})$. Then $-iJ_{\mathbb C}$ is a paracomplex structure on the complex Lie-Yamaguti algebra $(\g_{\mathbb C},[\cdot,\cdot]_{\mathbb C},\Courant{\cdot,\cdot,\cdot}_{\mathbb C})$, where $J_{\mathbb C}: \g_{\mathbb C} \to \g_{\mathbb C}$ is given by $$J_{\mathbb C}(x+iy)\triangleq Jx+iJy, \quad \forall x,y \in \g.$$
\end{cor}
\begin{pro}\label{complexproduct}\cite{Sheng Zhao}
Let $J$ be a complex structure and $E$ a product structure on a real Lie-Yamaguti algebra $(\g,[\cdot,\cdot],\Courant{\cdot,\cdot,\cdot})$. Then $\g$ has a complex product structure on $\g$ if and only if $\g$ has a complex structure $J$ and can be decomposed as $\g=\g_+\oplus \g_-$ such that
\begin{eqnarray*}
\g_-=J\g_+,
\end{eqnarray*}
where $\g_{\pm}$ are eigenspaces corresponding to the eigenvalues $\pm 1$ of $E$.
\end{pro}

Let $(\g,[\cdot,\cdot,\cdot],\Courant{\cdot,\cdot,\cdot })$ be a Lie-Yamaguti algebra and $E$ a paracomplex structure on $\g$. We define an endomorphism $J$ of $\g$  via an isomorphism $\phi:\g_+ \to \g_-$ by
\begin{eqnarray}
J(x+\alpha)=-\phi^{-1}(\alpha)+\phi(x), \quad \forall x\in \g_+, \alpha \in \g_-.\label{almostcom}
\end{eqnarray}
It is not hard to deduce that $J$ is an almost complex structure on $\g$ and $E\circ J=-J\circ E$. Moreover, we have the following proposition.
\begin{pro}
Let $E$ be a perfect paracomplex structure on a real Lie-Yamaguti algebra $(\g,[\cdot,\cdot],\Courant{\cdot,\cdot,\cdot})$. Then there is a perfect complex product structure on $\g$ if and only if there exists a linear isomorphism $\phi:\g_+\to \g_-$ satisfying the following equations
\begin{eqnarray}
\label{c.p}\phi([x,y])&=&-\phi^{-1}[\phi(x),y]-\phi^{-1}[x,\phi(y)]+[x,y],\\
~ \label{c.p2}\phi\Courant{x,y,z}&=&-\Courant{\phi(x),\phi(y),\phi(z)}+\Courant{\phi(x),y,z}+\Courant{x,\phi(y),z}+\Courant{x,y,\phi(z)}\\
~ &&\nonumber+\phi\Courant{\phi(x),\phi(y),z}+\phi\Courant{x,\phi(y),\phi(z)}+\phi\Courant{\phi(x),y,\phi(z)},\quad\forall x,y,z \in \g_+.
\end{eqnarray}
\end{pro}
\begin{proof}
Let $(J,E)$ be a perfect complex product structure on $\g$. Define a linear isomorphism $\phi:\g_+\longrightarrow\g_-$ by $\phi\triangleq J|_{\g_+}:\g_+\longrightarrow\g_-$. Then by Definition \ref{complex} and definition of perfect product structures in \cite{Sheng Zhao}, we deduce that Eqs. \eqref{c.p} and \eqref{c.p2} hold.

Conversely, define a linear map $J:\g\longrightarrow\g$ as in \eqref{almostcom}.
Then it is obvious that $J$ is an  almost complex structure on $\g$ and that $E\circ J=-J\circ E$. For all $\alpha,\beta,\gamma \in \g_-$, there exist $x,y,z\in \g_+$, such that $\phi(x)=\alpha,\phi(y)=\beta$ and $\phi(z)=\gamma$. By Eq. \eqref{c.p2}, we have
\begin{eqnarray*}
~ &&-\Courant{J\alpha,J\beta,J\gamma}+\Courant{J\alpha,\beta,\gamma}+\Courant{\alpha,J\beta,\gamma}+\Courant{\alpha,\beta,J\gamma}\\
~ &&+J\Courant{J\alpha,J\beta,\gamma}+J\Courant{\alpha,J\beta,J\gamma}+J\Courant{J\alpha,\beta,J\gamma}\\
~ &=&\Courant{x,y,z}-\Courant{x,\phi(y),\phi(z)}-\Courant{\phi(x),y,\phi(z)}-\Courant{\phi(x),\phi(y),z}\\
~ &&-\phi^{-1}\Courant{x,y,\phi(z)}-\phi^{-1}\Courant{\phi(x),y,z}-\phi^{-1}\Courant{x,\phi(y),z}\\
~ &=&-\phi^{-1}\Courant{\phi(x),\phi(y),\phi(z)}\\
~ &=&J\Courant{\alpha,\beta,\gamma},
\end{eqnarray*}
which implies that \eqref{ccom2} holds for all $\alpha,\beta,\gamma\in \g_-$. Similarly, by Eq. \eqref{c.p}, we obtain that
$$[J\alpha,J\beta]=J[J\alpha,\beta]+J[\alpha,J\beta]+[\alpha,\beta],\quad\forall \alpha,\beta\in \g_-.$$
Thus we obtain that $J$ is a complex structure on $\g_-$. Other cases can be deduced similarly. Thus $J$ is a complex structure, and hence $(J,E)$ is a perfect complex product structure on $\g$. This finishes the proof.
\end{proof}

In the sequel, we construct a perfect complex product structure via pre-Lie-Yamaguti algebras.
Let $(A,*,\{\cdot,\cdot,\cdot\})$ be a pre-Lie-Yamaguti algebra. A nondegenerate symmetric bilinear form $\frkB \in \otimes^2A^*$ on $A$ is called {\bf invariant} if for all $x,y,z,w \in A$
\begin{eqnarray}
\label{invcon1}\frkB(x*y,z)&=&-\frkB(y,x*z),\\
~ \label{invcon2}\frkB(\{x,y,z\},w)&=&\frkB(x,\{w,z,y\}).
\end{eqnarray}
Then $\frkB$ induces a linear isomorphism $\frkB^\sharp:A \to A^*$ by
\begin{eqnarray}
\langle\frkB^\sharp(x),y\rangle=\frkB(x,y), \quad \forall x,y \in A.\label{isom}
\end{eqnarray}
\begin{rmk}
It is clear that by a direct calculation, we have
\begin{eqnarray}
\label{invcon3}\frkB(\{x,y,z\}_D,w)&=&-\frkB(z,\{x,y,w\}_D).
\end{eqnarray}
\end{rmk}

\begin{thm}\label{complexpro}
Let $(A,*,\{\cdot,\cdot,\cdot\})$ be a pre-Lie-Yamaguti algebra with a nondegenerate symmetric invariant bilinear form $\frkB \in \otimes^2A^*$. Then there is a perfect complex product structure $(J,E)$ on the semidirect product Lie-Yamaguti algebra $A^c\ltimes_{L^*,-\huaR^*\tau} A^*$, where $E$ is given by \eqref{prod}, and the complex structure $J$ is given by
\begin{eqnarray}
J(x+\alpha)=-{\frkB^\sharp}^{-1}(\alpha)+\frkB^\sharp(x), \quad \forall x \in A, \alpha \in A^*.\label{com}
\end{eqnarray}
\end{thm}
\begin{proof}
By Proposition \ref{product}, $E$ is a perfect product on $A^c\ltimes_{L^*,-\huaR^*\tau} A^*$. For all $x,y \in A$, we have
\begin{eqnarray*}
~ &&[\frkB^\sharp(x),\frkB^\sharp(y)]_{L^*,-\huaR^*\tau}+(\frkB^\sharp)^{-1}([\frkB^\sharp(x),y]_{L^*,-\huaR^*\tau}+(\frkB^\sharp)^{-1}[x,\frkB^\sharp(y)]_{L^*,-\huaR^*\tau}
-[x,y]_C\\
~ &=&-(\frkB^\sharp)^{-1}\Big(L^*(y)\frkB^\sharp(x)\Big)+(\frkB^\sharp)^{-1}\Big(L^*(x)\frkB^\sharp(y)\Big)-[x,y]_C.
\end{eqnarray*}
For any $z \in A$, by \eqref{invcon1} and \eqref{isom}, we have
\begin{eqnarray*}
~ &&-\langle L^*(y)\frkB^\sharp(x)-L^*(x)\frkB^\sharp(y)+\frkB^\sharp([x,y]_C),z\rangle\\
~ &=&\langle\frkB^\sharp(x),y*z\rangle-\langle \frkB^\sharp(y),x*z\rangle+\langle\frkB^\sharp([x,y]_C),z\rangle\\
~ &=&\frkB(x,y*z)-\frkB (y,x*z)-\frkB(x*y,z)+\frkB(y*x,z)\\
~ &=&0,
\end{eqnarray*}
which implies that \eqref{c.p} holds.
Similarly, for all $x,y,z \in A$, by \eqref{invcon2}, \eqref{invcon3} and \eqref{isom}, we have
\begin{eqnarray*}
~ &&\footnotesize{-\Courant{\frkB^\sharp(x),\frkB^\sharp(y),\frkB^\sharp(z)}_{L^*,-\huaR^*\tau}+\Courant{\frkB^\sharp(x),y,z}_{L^*,-\huaR^*\tau}
+\Courant{x,\frkB^\sharp(y),z}_{L^*,-\huaR^*\tau}+\Courant{x,y,\frkB^\sharp(z)}_{L^*,-\huaR^*\tau}}\\
~ &&+\frkB^\sharp\Courant{\frkB^\sharp(x),\frkB^\sharp(y),z}_{L^*,-\huaR^*\tau}+\frkB^\sharp\Courant{x,\frkB^\sharp(y),\frkB^\sharp(z)}_{L^*,-\huaR^*\tau}
+\frkB^\sharp\Courant{\frkB^\sharp(x),y,\frkB^\sharp(z)}_{L^*,-\huaR^*\tau}\\
~ &=&-\huaR^*(z,y)\frkB^\sharp(x)+\huaR^*(z,x)\frkB^\sharp(y)+\huaL^*(x,y)\frkB^\sharp(z).
\end{eqnarray*}
And there holds that
\begin{eqnarray*}
\langle\frkB^\sharp(\Courant{x,y,z}_C),w\rangle&=&\langle\frkB^\sharp(\{x,y,z\}_D)+\frkB^\sharp(\{x,y,z\})-\frkB^\sharp(\{y,x,z\}),w\rangle\\
~ &=&\frkB(\{x,y,z\}_D,w)++\frkB(\{x,y,z\},w)-\frkB(\{y,x,z\},w)\\
~ &=&-\frkB(z,\{x,y,w\}_D)+\frkB(x,\{w,z,y\})-\frkB(y,\{w,z,x\})\\
~ &=&-\langle\frkB^\sharp(z),\{x,y,w\}_D\rangle+\langle\frkB^\sharp(x),\{w,z,y\}\rangle-\langle\frkB^\sharp(y),\{w,z,x\}\rangle\\
~ &=&\langle\huaL^*(x,y)\frkB^\sharp(z)-\huaR^*(z,y)\frkB^\sharp(x)+\huaR^*(z,x)\frkB^\sharp(y)\rangle,
\end{eqnarray*}
By arbitrary of $w$, we have
\begin{eqnarray*}
\frkB^\sharp(\Courant{x,y,z}_C)&=&-\Courant{\frkB^\sharp(x),\frkB^\sharp(y),\frkB^\sharp(z)}_{L^*,-\huaR^*\tau}+\Courant{\frkB^\sharp(x),y,z}_{L^*,-\huaR^*\tau}
+\Courant{x,\frkB^\sharp(y),z}_{L^*,-\huaR^*\tau}\\
~ &&+\Courant{x,y,\frkB^\sharp(z)}_{L^*,-\huaR^*\tau}+\frkB^\sharp\Courant{\frkB^\sharp(x),\frkB^\sharp(y),z}_{L^*,-\huaR^*\tau}
+\frkB^\sharp\Courant{x,\frkB^\sharp(y),\frkB^\sharp(z)}_{L^*,-\huaR^*\tau}\\
~ &&+\frkB^\sharp\Courant{\frkB^\sharp(x),y,\frkB^\sharp(z)}_{L^*,-\huaR^*\tau},
\end{eqnarray*}
which implies that \eqref{c.p2} holds. Thus $(J,E)$ is a perfect complex structure on $A^c\ltimes_{L^*,-\huaR^*\tau} A^*.$
\end{proof}

\begin{thm}
Let $(A,*,\{\cdot,\cdot,\cdot\})$ be a real pre-Lie-Yamaguti algebra. Then there is a perfect complex product structure on the real Lie-Yamaguti algebra $\mathfrak{aff}(A)=A^c\ltimes_{L,\huaR}A$.
\end{thm}
\begin{proof}
We define two endomorphisms $J$ and $E$ by
\begin{eqnarray*}
J(x,y)=(-y,x),~E(x,y)=(x,-y),\quad \forall x,y \in A.
\end{eqnarray*}
It is obvious that $E$ is a product structure on $\mathfrak{aff}(A)$ and that $J$ is an almost complex structure on $\mathfrak{aff}(A)$ satisfying $E\circ J=-J\circ E$. Clearly, $\mathfrak{aff}(A)_+=\{(x,0):x \in A\},~\mathfrak{aff}(A)_-=\{(0,x):x \in A\}$. Define $\phi:\mathfrak{aff}(A)_+ \to \mathfrak{aff}(A)_-$ by $\phi\triangleq J|_{\mathfrak{aff}(A)_+}:\mathfrak{aff}(A)_+ \to \mathfrak{aff}(A)_-$, i.e., $\phi(x,0)=(0,x)$. Then for all $(x,0),(y,0),(z,0) \in \mathfrak{aff}(A)_+ $, we have
\begin{eqnarray*}
~ &&[\phi(x,0),(y,0)]_{L,\huaR}+[(x,0),\phi(y,0)]_{L,\huaR}-\phi([(x,0),(y,0)]_{L,\huaR})-\phi^{-1}[\phi(x,0),\phi(y,0)]_{L,\huaR}\\
~ &=&[(0,x),(y,0)]_{L,\huaR}+[(x,0),(0,y)]_{L,\huaR}-\phi([(x,0),(y,0)]_{L,\huaR})\\
~ &=&(0,x*y-y*x)-\phi(x*y-y*x,0)\\
~ &=&0,
\end{eqnarray*}
which implies \eqref{c.p} holds.
Furthermore, we have
\begin{eqnarray*}
~ &&-\Courant{\phi(x,0),\phi(y,0),\phi(z,0)}_{L,\huaR}+\Courant{\phi(x,0),(y,0),(z,0)}_{L,\huaR}+\Courant{(x,0),\phi(y,0),(z,0)}_{L,\huaR}\\
~ &&+\Courant{(x,0),(y,0),\phi(z,0)}_{L,\huaR}+\phi\Courant{\phi(x,0),\phi(y,0),(z,0)}_{L,\huaR}+\phi\Courant{\phi(x,0),(y,0),\phi(z,0)}_{L,\huaR}\\
~ &&+\phi\Courant{(x,0),\phi(y,0),\phi(z,0)}_{L,\huaR}\\
~ &=&-\Courant{(0,x),(0,y),(0,z)}_{L,\huaR}+\Courant{(0,x),(y,0),(z,0)}_{L,\huaR}+\Courant{(x,0),\phi(y,0),(z,0)}_{L,\huaR}\\
~ &&+\Courant{(x,0),(y,0),(0,z)}_{L,\huaR}+\phi\Courant{(0,x),(0,y),(z,0)}_{L,\huaR}+\phi\Courant{(0,x),(y,0),(0,z)}_{L,\huaR}\\
~ &&+\phi\Courant{(x,0),(0,y),(0,z)}_{L,\huaR}\\
~ &=&(0,\{x,y,z\}_D+\{x,y,z\}-\{y,x,z\})\\
~ &=&\phi(\Courant{x,y,z}_C,0),
\end{eqnarray*}
from where \eqref{c.p2} follows. Thus $(J,E)$ is a perfect complex structure on $\mathfrak{aff}(A)=A^c\ltimes_{L,\huaR} A.$
\end{proof}

\section{Para-K\"{a}hler structures on Lie-Yamaguti algebras}
In this section, we add a compatibility condition between a symplectic structure and a paracomplex structure to introduce the notion of a para-K\"{a}hler structure on a Lie-Yamaguti algebra. A para-K\"{a}hler structure gives rise to a pseudo-Riemannian metric. Moreover, we introduce the notion of a Levi-Civita product associated to a pseudo-Riemannian Lie-Yamaguti algebra and give its relation with the associated pre-\LYA ~structure on a para-K\"ahler \LYA.

Recall that a {\bf phase space} of a \LYA ~$(\h,\br_{_\h},\ltp_\h)$ is a symplectic \LYA ~$\Big((T^*\h=\h\oplus\h^*,\br,\ltp),\omega_p\Big)$ such that $\h$ and $\h^*$ are subalgebras of $T^*\h$, where the symplectic structure $\omega_p$ is given by
\begin{eqnarray}
\omega_p(x+\alpha,y+\beta)=\pair{\alpha,y}-\pair{\beta,x},\quad\alpha,\beta\in \h^*,~x,y\in \h.\label{form}
\end{eqnarray}

In the sequel, we give the main definition in this section.
\begin{defi}
Let $(\g,[\cdot,\cdot],\Courant{\cdot,\cdot,\cdot})$ be a \LYA, $\omega$ a symplectic structure, and $E$ a paracomplex on $\g$. The pair $(\omega,E)$ is called a {\bf para-K\"{a}hler structre} on the Lie-Yamaguti algebra $\g$ if the following equality holds:
\begin{eqnarray}
\omega(Ex,Ey)=-\omega(x,y), \quad \forall x,y \in \g.\label{para}
\end{eqnarray}
The triple $(\g,\omega,E)$ is called a {\bf para-K\"{a}hler Lie-Yamaguti algebra}.
\end{defi}

We give an equivalent description of para-K\"{a}hler \LYA s.
\begin{thm}
Let $(\g,\omega)$ be a symplectic Lie-Yamaguti algebra. Then there is a paracomplex structure $E$ on $\g$ such that $(\g,\omega,E)$ is a para-K\"{a}hler Lie-Yamaguti algebra if and only if there exist two isotropic Lie-Yamaguti subalgebras $\g_+$ and $\g_-$ such that $\g=\g_+\oplus \g_-$ as the direct sum of vector spaces.
\end{thm}
\begin{proof}
Let $(\g,\omega,E)$ be a para-K\"{a}hler Lie-Yamaguti algebra. Since $E$ is a paracomplex structure on $\g$, we have $\g=\g_+\oplus\g_-$, where $\g_+$ and $\g_-$ are Lie-Yamaguti subalgebras of $\g$. For all $x,y \in \g_+$, on the one hand, we have
\begin{eqnarray*}
\omega(Ex,Ey)=\omega(x,y).\label{paraLY}
\end{eqnarray*}
On the other hand, since $(\g,\omega,E)$ is a para-K\"{a}hler Lie-Yamaguti algebra, by \eqref{para},
we get that $\g_+$ is isotropic. Similarly, $\g_-$ is also isotropic.

Conversely, since $\g_+$ and $\g_-$ are Lie-Yamaguti algebras, $\g=\g_+\oplus\g_-$ as vector spaces, there is a product structure on $\g$ defined by
\begin{eqnarray*}
E(x+\alpha)=x-\alpha,\quad \forall x \in \g, \alpha \in \g^*.
\end{eqnarray*}
Since $\g_+$ and $\g_-$ are isotropic, we have $dim(\g_+)=dim(\g_-)$. Thus $E$ is a paracomplex structure on $\g$. For all $x,y \in \g_+, \alpha,\beta \in \g_-$, since $\g_+$ and $\g_-$ are isotropic, we have
\begin{eqnarray*}
~ \omega(E(x+\alpha),E(y+\beta))&=&\omega(x-\alpha,y-\beta)=\omega(x,\beta)-\omega(\alpha,y)\\
~ &=&-\omega(x+\alpha,y+\beta).
\end{eqnarray*}
Thus $(\g,\omega,E)$ is a para-K\"{a}hler Lie-Yamaguti algebra. This finishes the proof.
\end{proof}

\begin{ex}\label{ex:parakaeler}
Let $(\g,\br,,\ltp)$ be a \LYA ~with a basis $\{e_1,e_2\}$ whose nonzero brackets are given as follows:
$$[e_1,e_2]=e_1,\quad \Courant{e_1,e_2,e_2}=e_1.$$
Then $(\omega,E)$ is a para-K\"ahler structure on $\g$, where $\omega$ and $E$ are given by\begin{eqnarray*}
\omega=ke^*_1\wedge e^*_2\quad{ and}\quad E=\begin{pmatrix}
 1 & 0 \\
 0 & -1
 \end{pmatrix}
\end{eqnarray*}
respectively.
\end{ex}

We construct a para-K\"ahler \LYA ~from a pre-\LYA.

\begin{pro}
Let $(A,*,\{\cdot,\cdot,\cdot\})$ be a pre-Lie-Yamaguti algebra. Then $(A^c\ltimes_{L^*,-\huaR^*\tau} A^*,\omega_p,E)$ is a perfect para-K\"{a}hler Lie-Yamaguti algebra, where $E$ is given by \eqref{prod}, and $\omega_p$ is given by \eqref{form}.
\end{pro}
\begin{proof}
By Theorem 4.7 in \cite{SZ1}, $(A^c\ltimes_{L^*,-\huaR^*\tau} A^*,\omega_p)$ is a symplectic Lie-Yamaguti algebra. By Proposition \ref{product}, $E$ is a perfect paracomplex structure on the phase space $T^*A^c$. For all $x,y \in A, ~ \alpha,\beta \in A^*$, we have
\begin{eqnarray*}
~ \omega_p(E(x+\alpha),E(y+\beta))&=&\omega_p(x-\alpha,y-\beta)=\langle-\alpha,y\rangle-\langle-\beta,x\rangle\\
~ &=&-\omega_p(x+\alpha,y+\beta).
\end{eqnarray*}
Hence, $(T^*A^c=A^c\ltimes_{L^*,-\huaR^*\tau} A^*,\omega_p,E)$ is a perfect para-K\"{a}hler Lie-Yamaguti algebra.
\end{proof}

\begin{pro}
Let $(\h,[\cdot,\cdot],\Courant{\cdot,\cdot,\cdot})$ be a Lie-Yamaguti algebra and $(\h\oplus\h^*,\omega_p)$ its phase space, where $\omega_p$ is given by \eqref{form}. Then $E:\h\oplus\h^*\to \h\oplus\h^*$ defined by
\begin{eqnarray*}
E(x+\alpha)=x-\alpha, \quad \forall x \in \h,\alpha\in \h^*,
\end{eqnarray*}
is a paracomplex structure and $(\h\oplus\h^*,\omega_p,E)$ is a para-K\"{a}hler Lie-Yamaguti algebra.
\end{pro}
\begin{proof}
It is obvious that $E^2=\Id$, and that $E$ satisfies the integrability condition, i.e. $E$ is a product structure on $\h\oplus \h^*$. Moreover, $\h$ and $\h^*$ have the same dimension, then $E$ is a paracomplex structure on $\h\oplus \h^*$. For all $x,y \in \h, \alpha,\beta \in \h^*$, we have
\begin{eqnarray*}
\omega_p(E(x+\alpha),E(y+\beta))=\omega_p(x-\alpha,y-\beta)=-\langle \alpha,y\rangle+\langle\beta, x\rangle=-\omega_p(x+\alpha,y+\beta).
\end{eqnarray*}
Thus $(\h\oplus\h^*,\omega_p,E)$ is a para-K\"{a}hler Lie-Yamaguti algebra.
\end{proof}
Let $(\g,\omega,E)$ be a para-K\"{a}hler Lie-Yamaguti algebra. Then it is obvious that $\g_-\cong\g_+^*$ via the symplectic structure $\omega$. Moreover, it is straightforward to deduce the following proposition.
\begin{pro}
Any para-K\"{a}hler Lie-Yamaguti algebra is isomorphic to a phase space of a Lie-Yamaguti algebra. More precisely, let $(\g,\omega,E)$ be a para-K\"{a}hler Lie-Yamaguti algebra such that $\g=\g_+\oplus\g_-$ , then it is isomorphic to the phase space of $\g_+$.
\end{pro}

In the sequel, we focus on the para-K\"{a}hler Lie-Yamaguti algebras whose paracomplex structures are perfect and abelian at the same time.
\begin{pro}
Let $(\g,\omega,E)$ be a para-K\"{a}hler Lie-Yamaguti algebra and $E$ be abelian and perfect. The pre-Lie-Yamaguti algebra structure defined by $\omega$ verifies
\begin{eqnarray}
~\label{a1} E(x*y)&=&Ex*y,\\
~ \label{a2}E\{x,y,z\}&=&-\{Ex,y,z\}+\{x,Ey,z\}+\{x,y,Ez\}.
\end{eqnarray}
\end{pro}
\begin{proof}
For all $x,y,z,w\in \g$, by Proposition 4.3 in \cite{SZ1}, we have
\begin{eqnarray*}
\omega\Big(E(x*y),z\Big)=-\omega\Big(x*y,Ez\Big)=\omega\Big(y,[x,Ez]\Big)\stackrel{\text{$E$ is abelian}}{=}-\omega\Big(y,[Ex,z]\Big)=\omega\Big(Ex*y,z\Big),
\end{eqnarray*}
which proves \eqref{a1}.
And similarly, we have
\begin{eqnarray*}
\omega\Big(E\{x,y,z\},w\Big)&=&-\omega\Big(\{x,y,z\},Ew\Big)=-\omega\Big(x,\Courant{Ew,z,y}\Big)\\
~ &=&\omega\Big(x,\Courant{w,Ez,y}\Big)+\omega\Big(x,\Courant{w,z,Ey}\Big)+\omega\Big(x,\Courant{Ew,Ez,Ey}\Big)~~(\text{Since $E$ is abelian})\\
~ &=&\omega\Big(x,\Courant{w,Ez,y}\Big)+\omega\Big(x,\Courant{w,z,Ey}\Big)-\omega\Big(Ex,\Courant{w,z,y}\Big)~~(\text{Since $E$ is perfect})\\
~ &=&\omega\Big(\{x,y,Ez\}+\{x,Ey,z\}-\{Ex,y,z\},w\Big),
\end{eqnarray*}
which proves \eqref{a2}. This finishes the proof.
\end{proof}

\begin{rmk}
By the definition of $\{\cdot,\cdot,\cdot\}_D$ and \eqref{a2}, we can get the following directly
\begin{eqnarray*}
E\{x,y,z\}_D=\{Ex,y,z\}_D+\{x,Ey,z\}_D-\{x,y,Ez\}_D.
\end{eqnarray*}
\end{rmk}

\begin{pro}
Let $(\g,\omega,E)$ be a para-K\"{a}hler Lie-Yamaguti algebra and $E$ be abelian. Then abelian Lie-Yamaguti subalgebras $\g_+$ and $\g_-$ are pre-Lie-Yamaguti subalgebras of $\g$ endowed with the pre-Lie-Yamaguti algebra structures defined by $\omega$.
\end{pro}
\begin{proof}
For all $x,y,z,w \in \g_+$, one has
\begin{eqnarray*}
\omega(x*y,z)=-\omega(y,[x,z]_{\g_+})=0,
\end{eqnarray*}
and one has
\begin{eqnarray*}
\omega(\{x,y,z\},w)=\omega(x,\Courant{w,z,y}_{\g_+})=0,
\end{eqnarray*}
hence $x*y,~\{x,y,z\} \in \g_+$, which shows that $\g_+$ is a subalgebra for the pre-Lie-Yamaguti algebra structures. An analogous reason shows the result for $\g_-$.
\end{proof}

At the end of this section, we study the Levi-Civita product associated to a para-K\"{a}hler Lie-Yamaguti algebra.
\begin{defi}
A {\bf pseudo-Riemannian Lie-Yamaguti algebra} is a Lie-Yamaguti algebra $(\g,[\cdot,\cdot],\\
\Courant{\cdot,\cdot,\cdot})$ endowed with a nondegenerate symmetric bilinear form (also called the pseudo-Riemannian metric) $S$. The associated {\bf Levi-Civita product} are a pair of products $(\nabla,\Delta)$ on $\g$: $ \nabla:\otimes^2\g \to \g$ with $(x,y)\mapsto \nabla_xy$, and $\Delta:\otimes^3\g\to \g$ with $(x,y,z)\mapsto \Delta_{x,y} z$ respectively given by for all $x,y,z,w \in \g$
\begin{eqnarray}
~ S(\nabla_xy,z)&=&S([x,y],z)+S([z,x],y)+S([z,y],x),\\
~ 3S(\Delta_{x,y} z,w)&=&S(\Courant{x,y,w},z)+S(\Courant{x,y,z},w)+S(\Courant{w,z,x},y)+2S(\Courant{w,z,y},x).
\end{eqnarray}
\end{defi}
\emptycomment{
\begin{pro}
Let $\Big((\g,[\cdot,\cdot],\Courant{\cdot,\cdot,\cdot}),S\Big)$ be a pseudo-Riemannian Lie-Yamaguti algebra. If $S$ satisfies the invariant condition \eqref{invr1} and \eqref{invr2} ($(\g,S)$ is now a quadratic  Lie-Yamaguti algebra), then the Levi-Civita products satisfy the following equalities:
\begin{eqnarray}
~\label{lc1}\nabla_xy&=&{1\over3}[x,y],\\
~ \label{lc3}\Delta_{x,y} z-{1\over2}\Delta_{y,x} z&=&{1\over2}\Courant{x,y,z}.
\end{eqnarray}
\end{pro}
\begin{proof}
For all $x,y \in \g$, by the invariant condition \eqref{invr1}, we have
\begin{eqnarray*}
~  S(\nabla_xy,z)&=&S([x,y],z)+S([z,x],y)+S([y,z],x)\\
~ &=&3S([x,y],z),
\end{eqnarray*}
which implies that \eqref{lc1} holds.
Moreover, if $S$ is invariant, for all $x,y,z,w \in \g$, by \eqref{invr2} and \eqref{invr3}, we have
\begin{eqnarray*}
~ 3S(\Delta_{x,y} z,w)&=&S(\Courant{w,z,x},y)+S(\Courant{x,y,z},w)-S(\Courant{w,z,y},x)-2S(\Courant{x,y,w},z).
\end{eqnarray*}
And we have that
\begin{eqnarray*}
~ 3S(\Delta_{x,y} z,w)&=&S(\Courant{x,y,w},z)+S(\Courant{x,y,z},w)+S(\Courant{w,z,x},y)+2S(\Courant{w,z,y},x),\\
~ -3S(\Delta_{y,x} z,w)&=&-S(\Courant{y,x,w},z)-S(\Courant{y,x,z},w)-S(\Courant{w,z,y},x)-2S(\Courant{w,z,x},y).
\end{eqnarray*}
Adding these three equations yields that
\begin{eqnarray*}
3S(2\Delta_{x,y} z-\Delta_{y,x} z,w)=3S(\Courant{x,y,z},w),
\end{eqnarray*}
which implies that \eqref{lc3} holds.
This finishes the proof.
\end{proof}}

\begin{pro}
Let $(\g,\omega,E)$ be a para-K\"{a}hler Lie-Yamaguti algebra. Define a bilinear form $S$ on $\g$ by
\begin{eqnarray*}
S(x,y)\triangleq\omega(x,Ey), \quad \forall x,y\in \g.
\end{eqnarray*}
Then $(\g,S)$ is a pseudo-Riemannian Lie-Yamaguti algebra. Moreover, the associated Levi-Civita product $\Delta$ and the perfect paracomplex structure $E$ satisfy the following compatibility condition:
\begin{eqnarray*}
E\Delta_{x,y} z=\Delta_{Ex,Ey} Ez.
\end{eqnarray*}
\end{pro}
\begin{proof}
Since $\omega$ is skew-symmetric and $\omega(Ex,Ey)=-\omega(x,y)$, we have
\begin{eqnarray*}
S(y,x)=\omega(y,Ex)=-\omega(Ey,E^2x)=-\omega(Ey,x)=\omega(x,Ey)=S(x,y),
\end{eqnarray*}
which implies that $S$ is symmetric. It is obvious that $S$ is nondegenerate, thus $S$ is a pseudo-Riemannian metric on $\g$. Moreover, by \eqref{paraLY}, we have
\begin{eqnarray*}
S(Ex,y)=\omega(Ex,Ey)=-\omega(x,y)=-S(x,Ey).
\end{eqnarray*}
Thus sine $E$ is perfect, we have
\begin{eqnarray*}
~ && 3S(\Delta_{Ex,Ey} Ez,w)\\
~ &=&S(\Courant{Ex,Ey,w},Ez)+S(\Courant{Ex,Ey,Ez},w)+S(\Courant{w,Ez,Ex},Ey)+2S(\Courant{w,Ez,Ey},Ex)\\
~ &=&S(E\Courant{x,y,Ew},Ez)+S(E\Courant{x,y,z},Ew)+S(E\Courant{Ew,z,x},Ey)+2S(E\Courant{Ew,z,y},Ex)\\
~ &=&-\Big(S(\Courant{x,y,Ew},z)+S(\Courant{x,y,z},Ew)+S(\Courant{Ew,z,x},y)+2S(\Courant{Ew,z,y},x)\Big)\\
~ &=&-3S(\Delta_{x,y} z,Ew)\\
~ &=&3S(E\Delta_{x,y} z,w),
\end{eqnarray*}
which implies that $E\Delta_{x,y} z=\Delta_{Ex,Ey} Ez$ holds.
\end{proof}

 There is a close relationship between the Levi-Civita product on a para-K\"{a}hler Lie-Yamaguti algebra and its compatible pre-Lie-Yamaguti algebra structure. The following proposition reinforces this fact.
\begin{pro}
 Let $(\g,\omega,E)$ be a para-K\"{a}hler Lie-Yamaguti algebra and $(\nabla,\Delta)$ the associated Levi-Civita product. Then for all $x,y,z\in \g_+$ and $\alpha,\beta,\gamma \in \g_-$, we have
\begin{eqnarray*}
\nabla_xy= x*y,~ \Delta_{x,y} z=\{x,y,z\},
~\nabla_\alpha\beta= \alpha*\beta,~ \Delta_{\alpha,\beta}\gamma=\{\alpha,\beta,\gamma\}.
\end{eqnarray*}
Moreover, if $S$ is invariant, we also have
\begin{eqnarray*}
~\Delta_{x,y} z=\{x,y,z\}_D,~\Delta_{\alpha,\beta}\gamma=\{\alpha,\beta,\gamma\}_D.
\end{eqnarray*}
\end{pro}
\begin{proof}
Since $(\g,\omega,E)$ is a para-K\"{a}hler Lie-Yamaguti algebra and $\g=\g_+\oplus\g_-$, where $\g_+$ and $\g_-$ are isotropic subalgebras, then for all $x,y,z,w \in \g_+$, we have that
\begin{eqnarray*}
\omega(\nabla_xy,z)=0,~\omega(\Delta_{x,y} z,w)=0.
\end{eqnarray*}
Since $\g_+$ is isotropic, we obtain $\nabla_xy,\Delta_{x,y} z \in\g_+$. Similarly, for all $\alpha,\beta,\gamma \in \g_-$, we have
$\nabla_\alpha\beta,\Delta_{\alpha,\beta} \gamma \in \g_-$. Furthermore, for all $x,y \in \g_+, \alpha \in \g_-$, we have
\begin{eqnarray*}
~ &&\omega(\nabla_xy,\alpha)=S(\nabla_xy,E\alpha)=-S(\nabla_xy,\alpha)\\
~ &=&-S([x,y],\alpha)-S([\alpha,x],y)-S([\alpha,y],x)\\
~ &=&-\omega(\alpha,[x,y])+\omega(y,[\alpha,x])+\omega(x,[\alpha,y])\\
~ &=&-\omega(y,[x,\alpha])=\omega(x*y,\alpha).\\
\end{eqnarray*}
For all $x,y,z \in \g_+,\alpha \in\g_-$, we have
\begin{eqnarray*}
~ && 3\omega(\Delta_{x,y} z,\alpha)=3S(\Delta_{x,y} z,E\alpha)=-3S(\Delta_{x,y} z,\alpha)\\
~ &=&-S(\Courant{x,y,\alpha},z)-S(\Courant{x,y,z},\alpha)-S(\Courant{\alpha,z,x},y)-2S(\Courant{\alpha,z,y},x)\\
~ &=&\omega(z,\Courant{x,y,\alpha})-\omega(\alpha,\Courant{x,y,z})+\omega(y,\Courant{\alpha,z,x})+2\omega(x,\Courant{\alpha,z,y})\\
~ &=&3\omega(x,\Courant{\alpha,z,y})\\
~ &=&3\omega(\{x,y,z\},\alpha).
\end{eqnarray*}
Thus we have proved that
\begin{eqnarray*}
\nabla_xy=x*y,~ \Delta_{x,y} z=\{x,y,z\}.
\end{eqnarray*}
Furthermore, if $S$ is invariant, we have
\begin{eqnarray*}
~ &&3\omega(\Delta_{x,y} z,\alpha)=3S(\Delta_{x,y} z,E\alpha)=-3S(\Delta_{x,y} z,\alpha)\\
~ &=&S(\Courant{\alpha,z,y},x)-S(\Courant{\alpha,z,x},y)-S(\Courant{x,y,z},\alpha)+2S(\Courant{x,y,\alpha},z)\\
~ &=&-\omega(x,\Courant{\alpha,z,y})+\omega(y,\Courant{\alpha,z,x})-\omega(\alpha,\Courant{x,y,z})-2\omega(z,\Courant{x,y,\alpha})\\
~ &=&-3\omega(z,\Courant{x,y,\alpha})=3\omega(\{x,y,z\}_D,\alpha),
\end{eqnarray*}
which implies that $~\Delta_{x,y} z=\{x,y,z\}_D$.
Other equalities can be proved similarly and we omit the details.
\end{proof}

\section{Pseudo-K\"{a}hler structures on Lie-Yamaguti algebras}
In this section, we add a compatibility condition between a symplectic structure and a complex structure to introduce the notion of a pseudo-K\"{a}hler structure on a Lie-Yamaguti algebra. Moreover, the relation between para-K\"{a}hler structures and pseudo-K\"{a}hler structures on a Lie-Yamaguti algebra is studied, and we construct a K\"ahler \LYA ~from a pre-\LYA.

\begin{defi}
Let $(\g,[\cdot,\cdot],\Courant{\cdot,\cdot,\cdot})$ be a a real Lie-Yamaguti algebra, $\omega$ a symplectic structure, and $J$ a complex structure on $\g$. The pair $(\omega,J)$ is called a {\bf pseudo-K\"{a}hler structure} on the Lie-Yamaguti algebra $\g$ if
\begin{eqnarray}
\omega(Jx,Jy)=\omega(x,y), \quad \forall x,y \in \g.\label{pse}
\end{eqnarray}
The triple $(\g,\omega,J)$ is called a real {\bf pseudo-K\"{a}hler Lie-Yamaguti algebra}.
\end{defi}

\begin{ex}
Let $(\g,\br,,\ltp)$ be the \LYA ~given in Example \ref{ex:parakaeler}, then $(\omega,J)$ is a pseudo-K\"ahler structure on $\g$, where $\omega$ and $J$ are given by
\begin{eqnarray*}
\omega=ke_1^*\wedge e_2^*\quad and \quad J=\begin{pmatrix}
0 & 1\\
-1 & 0
\end{pmatrix}
\end{eqnarray*}
respectively.
\end{ex}

The following two theorems illustrate the relation between the para-K\"{a}hler structures and pseudo-K\"{a}hler structures on Lie-Yamaguti algebras.

\begin{thm}
Let $(\g,\omega,E)$ be a complex para-K\"{a}hler Lie-Yamaguti algebra. Then $(\g_{\mathbb R},\omega_{\mathbb R},J)$ is a real pseudo-K\"{a}hler Lie-Yamaguti algebra, where $\g_{\mathbb R}$ is the underlying real Lie-Yamaguti algebra, $J=iE$ and $\omega_{\mathbb R}={\Rea}(\omega)$ is the real part of $\omega$.
\end{thm}
\begin{proof}
By Proposition \ref{complex1}, $J=iE$ is a complex structure on the complex Lie-Yamaguti algebra $\g$. Thus $J$ is also a complex structure on the real Lie-Yamaguti algebra $\g_{\mathbb R}$. It is obvious that $\omega_{\mathbb R}$ is skew-symmetric. If for all $x \in \g$, $\omega_{\mathbb R}(x,y)=0$. Then we have
\begin{eqnarray*}
\omega(x,y)=\omega_{\mathbb R}(x,y)+i\omega_{\mathbb R}(-ix,y)=0.
\end{eqnarray*}
By the nondegeneracy of $\omega$, we obtain that $y=0$. Thus
$\omega_{\mathbb R}$  is nondegenerate. Therefore, $\omega_{\mathbb R}$ is a symplectic structure on the real Lie-Yamaguti algebra $\g_{\mathbb R}$. By $\omega(Ex,Ey)=-\omega(x,y)$, we have
\begin{eqnarray*}
\omega_{\mathbb R}(Jx,Jy)={\Rea}(\omega(iEx,iEy))={\Rea}(-\omega(Ex,Ey))={\Rea}(\omega(x,y))=\omega_{\mathbb R}(x,y).
\end{eqnarray*}
Thus $(\g_{\mathbb R},\omega_{\mathbb R},J)$ is a real pseudo-K\"{a}hler Lie-Yamaguti algebra.
\end{proof}

Conversely, we have the following theorem.
\begin{thm}
Let $(\g,\omega,J)$ be a real pseudo-K\"{a}hler Lie-Yamaguti algebra. Then $(\g_{\mathbb C},\omega_{\mathbb C},E)$ is a complex para-K\"{a}hler Lie-Yamaguti algebra, where $\g_{\mathbb  C}=\g\otimes_{\mathbb R}\mathbb C$  is the complexification of $\g$, $E=-iJ_{\mathbb C}$ and $\omega_{\mathbb C}$ is the complexification of $\omega$:
\begin{eqnarray*}
\omega_{\mathbb C}(x_1+iy_1,x_2+iy_2)=\omega(x_1,x_2)-\omega(y_1,y_2)+i\omega(x_1,y_2)+i\omega(y_1,x_2),\\
\nonumber\quad \forall x_1,x_2,y_1,y_2 \in \g.
\end{eqnarray*}
\end{thm}
\begin{proof}
By Corollary \ref{complex2}, $E=-iJ_{\mathbb C}$ is a paracomplex structure on the complex Lie-Yamaguti algebra $\g_{\mathbb C}$. It is obvious that $\omega_{\mathbb C}$ is skew-symmetric and nondegenerate. Moreover, since $\omega$ is a symplectic structure on $\g$, we deduce that $\omega_{\mathbb C}$ is a symplectic structure on $\g_{\mathbb C}$. Finally, by $\omega(Jx,Jy)=\omega(x,y)$, we have
\begin{eqnarray*}
\omega_{\mathbb C}(E(x_1+iy_1),E(x_2+iy_2))&=&\omega_{\mathbb C}(Jy_1-iJx_1,Jy_2-iJx_2)\\
~ &=&\omega(Jy_1,Jy_2)-\omega(Jx_1,Jx_2)-i\omega(Jx_1,Jy_2)-i\omega(Jy_1,Jx_2)\\
~ &=&\omega(y_1,y_2)-\omega(x_1,x_2)-i\omega(x_1,y_2)-i\omega(y_1,x_2)\\
~ &=&-\omega_{\mathbb C}(x_1+iy_1,x_2+iy_2).
\end{eqnarray*}
Thus $(\g_{\mathbb C},\omega_{\mathbb C},-iJ_{\mathbb C})$ is a complex para-K\"{a}hler Lie-Yamaguti algebra.
\end{proof}

\begin{pro}
Let $(\g,\omega,J)$ be a real pseudo-K\"{a}hler Lie-Yamaguti algebra. Define a bilinear form $S$ on $\g$ by
\begin{eqnarray}
S(x,y)\triangleq\omega(x,Jy), \quad \forall x,y \in \g.\label{formcom}
\end{eqnarray}
Then $(\g,S)$ is a pseudo-Riemannian Lie-Yamaguti algebra.
\end{pro}
\begin{proof}
By \eqref{pse}, we have
\begin{eqnarray*}
S(y,x)=\omega(y,Jx)=\omega(Jy,J^2x)=-\omega(Jy,x)=\omega(x,Jy)=S(x,y),
\end{eqnarray*}
which implies that $S$ is symmetric. Moreover, since $\omega$ is nondegenerate and $J^2=-\Id$, it is obvious that $S$ is nondegenerate. Thus $S$ is a pseudo-Riemannian metric on the Lie-Yamaguti algebra $\g$.
\end{proof}
\begin{defi}
Let $(\g,\omega,J)$ be a real pseudo-K\"{a}hler Lie-Yamaguti algebra. If the associated pseudo-Riemannian metric defined by \eqref{formcom} is positive definite, we call $(\g,\omega,J)$ a real {\bf K\"{a}hler Lie-Yamaguti algebra}.
\end{defi}

At the end of this section, we construct a K\"{a}hler Lie-Yamaguti algebra from a pre-Lie-Yamaguti algebra with a symmetric and invariant bilinear form.
\begin{pro}
Let $(A,*,\{\cdot,\cdot,\cdot\})$ be a real pre-Lie-Yamaguti algebra with a symmetric invariant bilinear form $\frkB$. Then $(A^c\ltimes_{L^*,-\huaR^*\tau}A^*, \omega_p,-J)$ is a real K\"{a}hler Lie-Yamaguti algebra, where $J$ is given by \eqref{com} and $\omega_p$ is given by \eqref{form}.
\end{pro}
\begin{proof}
By Proposition \ref{product} and \ref{complexpro}, we have that $\omega_p$ is a symplectic structure and $J$ is a complex structure on the semidirect product Lie-Yamaguti algebra $A^c\ltimes_{L^*,-\huaR^*\tau}A^*$. Obviously, $-J$ is a complex structure on $A^c\ltimes_{L^*,-\huaR^*\tau}A^*$. Let $\{e_1,\cdots,e_n\}$ be a basis of $A$ such that $\frkB(e_i,e_j)=\delta_{ij}$ and $\{e_1^*,\cdots,e_n^*\}$ be the dual basis of $A^*$. Then for all $i,j,k,l
\in \{1,\cdots,n\}$, we have
\begin{eqnarray*}
\omega_p(e_i+e_j^*,e_k+e_l^*)&=&\delta_{jk}-\delta_{li},\\
~ \omega_p(-J(e_i+e_j^*),-J(e_k+e_l^*))&=&\omega_p(e_j-e_i^*,e_l-e_k^*)=-\delta_{il}+\delta_{kj}
\end{eqnarray*}
which implies that
\begin{eqnarray*}
\omega_p(-J(x+\alpha),-J(y+\beta))=\omega_p(x+\alpha,y+\beta), \quad \forall x,y \in A, \alpha,\beta \in A^*.
\end{eqnarray*}
Therefore $(A^c\ltimes_{L^*,-\huaR^*\tau}A^*,\omega_p,-J)$ is a pseudo-K\"{a}hler Lie-Yamaguti algebra. Finally, let $x=\sum_{i=1}^n\lambda_i e_i \in A, ~ \alpha=\sum_{j=1}^n\mu_j e_j^* \in A^*$ such that $x+\alpha \neq 0$. We have
\begin{eqnarray*}
~ S(x+\alpha,x+\alpha)&=&\omega_p(x+\alpha,-J(x+\alpha))\\
~ &=&\omega_p\Big(\sum_{i=1}^n\lambda_i e_i+\sum_{j=1}^n\mu_j e_j^*,\sum_{j=1}^n\mu_j e_j-\sum_{i=1}^n\lambda_i e_i^*\Big)\\
~ &=&\sum_{j=1}^n \mu_j^2+\sum_{i=1}^n \lambda_i^2>0.
\end{eqnarray*}
Therefore, $S$ is positive definite. Thus $(A^c\ltimes_{L^*,-\huaR^*\tau}A^*, \omega_p,-J)$ is a real K\"{a}hler Lie-Yamaguti algebra. This finishes the proof.
\end{proof}

\end{document}